\documentclass[11pt,oneside,reqno]{amsart}
\usepackage{amsmath,amssymb,amsthm}
\usepackage{arydshln}
\usepackage[english]{babel}
\usepackage{pifont, verbatim}
\usepackage{mathrsfs}
\usepackage{tikz-cd}
\usepackage{stmaryrd}
 2

\usepackage[colorlinks=true,linkcolor=blue,citecolor=blue,pdfpagelabels=false]{hyperref}
\newtheorem{thm}{Theorem}[section]
\newtheorem{lem}[thm]{Lemma}
\newtheorem{prop}[thm]{Proposition}

\newtheorem{cor}[thm]{Corollary}
\newtheorem{defn}[thm]{Definition}
\newcommand{\thmref}[1]{Theorem~\ref{#1}}
\newcommand{\defref}[1]{Definition~\ref{#1}}
\newcommand{\lemref}[1]{Lemma~\ref{#1}}

\newcommand{\propref}[1]{Proposition~\ref{#1}}
\newcommand{\corref}[1]{Corollary~\ref{#1}}

\newtheorem{rmk}[thm]{Remark}

\newenvironment{acknowledgements}{\bigskip\textbf{Acknowledgements.}}{}

\newcommand{\nequiv}{\not\equiv}
\renewcommand{\geq}{\geqslant}

\begin{document}

\baselineskip=17pt

\title[Reductions of Galois representations]
{Reductions of Galois representations and the Theta operator}
\author[E. Ghate and A. Kumar]{Eknath Ghate and Arvind Kumar}

\address[Eknath Ghate]{School of Mathematics, Tata Institute of Fundamental Research, Homi Bhabha Road, Mumbai 400005, India.}
\email{eghate@math.tifr.res.in}

\address[Arvind Kumar]
{Einstein Institute of Mathematics, The Hebrew University of Jerusalem, Edmund Safra Campus, Jerusalem 91904, Israel. \newline {\em Current Address: } 
{Department of Mathematics, Indian Institute of Technology Jammu, Jagti NH-44, PO Nagrota, Jammu 181221, India.}}
\email{arvind.kumar@iitjammu.ac.in}



\begin{abstract}
Let $p\ge 5$ be a prime, and let $f$ be a cuspidal eigenform of weight at least $2$ and level coprime to $p$ of finite slope $\alpha$.
Let $\bar{\rho}_f$ denote the mod $p$ Galois representation associated with $f$ and $\omega$  
the mod $p$ cyclotomic character. Under an assumption on the weight of $f$, we prove that there exists a cuspidal eigenform  
$g$ of weight at least $2$ and level coprime to $p$ of slope $\alpha+1$ such that 
$$\bar{\rho}_f \otimes \omega \simeq \bar{\rho}_g,$$ up to semisimplification.
The proof uses Hida-Coleman families and the theta operator acting on 
overconvergent forms. The structure of the reductions of the local Galois representations 
associated to cusp forms with slopes in the interval $[0,1)$ were determined by Deligne, Buzzard and Gee and for slopes in
 $[1,2)$  by Bhattacharya, Ganguli, Ghate, Rai and Rozensztajn. We show that these reductions, in spite of their somewhat complicated 
behavior, are compatible with the displayed equation above. Moreover, the displayed equation above allows us to predict the shape 
of the reductions of a class of Galois representations attached to eigenforms of slope larger than $2$. 
Finally, the methods of this paper allow us to obtain upper bounds on the radii of certain Coleman families. 
\end{abstract}
\maketitle

\section{Introduction}
Let $p$ be a prime and let $\bar{\rho}:{\rm Gal}(\bar{\mathbb{Q}}/\mathbb{Q}) \rightarrow {\rm GL}_2(\overline{\mathbb{F}}_p)$ 
be a continuous, absolutely irreducible, two-dimensional, odd, mod $p$ Galois representation. Such a representation 
is said to be of Serre-type. Serre's modularity conjecture in its qualitative form claims that every $\bar{\rho}$ of Serre-type 
is of the form $\bar\rho_f$ for some eigenform $f$. The refined or quantitative form of the conjecture also specifies a minimal 
weight $k(\bar{\rho})$, known as the Serre weight, and a level $N(\bar{\rho})$ for $f$. The level $N(\bar{\rho})$ 
is taken to be the Artin conductor of $\bar{\rho}$ outside $p$, whereas the weight $k(\bar{\rho})$ is built out of information on 
the ramification of $\bar\rho$ at $p$. This conjecture is now a theorem due to the celebrated work of 
Khare \cite{kha}, Khare-Wintenberger \cite{khw} and Kisin \cite{kis}.
	
Now suppose $f$ is an eigenform of weight $k \ge 2$ with $\bar{\rho}_f$ absolutely irreducible. Let $\omega$ denote the mod $p$ cyclotomic character. Since 
$\bar\rho_f$ is of Serre-type, the twisted representation $\bar{\rho}_f \otimes \omega$ is again of Serre-type, hence by Serre's modularity conjecture, it 
arises from an eigenform, say $g$, i.e., 
$$
\bar{\rho}_f \otimes \omega \simeq \bar{\rho}_g.
$$
Serre's conjecture gives the minimal weight, level and character of the eigenform $g$ from the corresponding data for $f$. It would 
be interesting to investigate how the slope of $g$ depends on the slope of $f$. The conjecture does not give 
any information about this. In many cases, computational evidence suggests that
\begin{eqnarray}
 \label{slope f vs g}
    \text{ slope of } g & = & \text{ slope of } f + 1,
\end{eqnarray}
if $f$ and $g$ are normalized to have first Fourier coefficient $1$.
But in fact \eqref{slope f vs g} is not always true, see Sec. \ref{serexmp} for some examples. 
However, we prove the following general result. 

\begin{thm}\label{main1}
Let $p \ge 5$ be a prime and $N$ be a positive integer such that $(p,N)=1$. Suppose that $f\in S_k(N,\chi)$ is an eigenform of weight $k \ge 2$, level $N$,
character $\chi$ and slope $\alpha$ with a $p$-stabilization $f_k$ of slope $\alpha$, and suppose 
$\bar{\rho}_f$ denotes the mod $p$ Galois representation associated with $f$. Let $M_{f_{k}}$ 
be the non-negative integer in Definition~\ref{radius} 
and let $\delta_{f_k}$ be the Kronecker delta function defined in \eqref{delta}.  
If $k \equiv 2-\kappa  \mod p^{M_{{f_k}} + \delta_{f_k}}$  with $\kappa \in \{2,3, \dots, p^{M_{f_{k}} + \delta_{f_k}}+1 \}$, 
then there is an eigenform $g\in S_l(N,\chi)$ of slope $\alpha+ \kappa- 1$ such that
\begin{equation}\label{Maineqn}
 \bar{\rho}_f \otimes \omega^{\kappa-1} \simeq \bar{\rho}_g,
\end{equation}
up to semisimplification.
Moreover, if $f$ is a newform, we may choose $g$ to be a newform.
\end{thm}

\noindent If the slope $\alpha$ of $f$ is smaller than $\frac{k-1}{2}$, there is always a $p$-stabilization $f_k$ of $f$ of slope $\alpha$, so in many cases
the hypothesis imposed on the $p$-stabilization in Theorem~\ref{main1} holds automatically. Also, there is (another) non-negative 
integer $M$ 
such that the weight $l$ of $g$ in Theorem~\ref{main1} can be chosen to be 
any integer satisfying the following conditions:
\begin{enumerate}
 \item [(i)] $l > 2\alpha +2\kappa,$
\item[{(ii)}] $l =(k-2+\kappa)p^{M}+\kappa +n(p-1)p^{M}$, for any $n\in \mathbb{Z}$.
 \end{enumerate}
 
The simplest case of the theorem above is the case $\kappa=2$, which is of special interest.
\begin{cor}\label{main}
Let $p \ge 5$ be a prime and $N$ be a positive integer such that $(p,N)=1$. 
Suppose that $f\in S_k(N,\chi)$ is an eigenform of finite slope $\alpha$ as in Theorem~\ref{main1}. 
If $\alpha > 0$,  assume that $k \equiv 0 \mod p^{M_{f_k} + \delta_{f_k}}$. 
Then there is an eigenform $g\in S_l(N,\chi)$ of slope $\alpha +1$ such that
\begin{equation}\label{maineqn}
 \bar{\rho}_f \otimes \omega \simeq \bar{\rho}_g,
\end{equation}
up to semisimplification.
\end{cor}

Let $f$ be as in \corref{main}. Then there are three forms satisfying condition \eqref{maineqn}, which are natural choices for the form $g$ in \corref{main}. These are
\begin{enumerate}
 \item[(a)]
 $\theta f$, where the theta operator $\theta=q\frac{d}{dq}$ is defined on $q$-expansions by
\begin{equation}\label{theta}
\theta\left( \sum_{n = 0}^\infty  a_nq^n  \right)= \sum_{n = 0}^\infty na_nq^n,
\end{equation}
 \item[(b)] 
 a minimal weight form associated to $\bar{\rho}_f \otimes \omega$ by Serre's conjecture if $\bar\rho_f$ is irreducible,
 \item[(c)]
$f_\omega := f \otimes \omega$, the twist of $f$ by the (Teichm\"uller lift of the) character $\omega$. 
\end{enumerate}
But in  \corref{main}, we cannot simply take $g$ to be any of the forms above. The form $\theta f$ in (a) has the right slope 
$\alpha+1$ but it is not a classical eigenform 
because, e.g., the local Galois representation 
corresponding to $\theta f$ is still crystalline but has  Hodge-Tate weights $(1,k)$ instead of $(0,l)$, for some integer $l$.
Furthermore, the form obtained from Serre's conjecture in (b) and the twisted form $f_\omega$ in (c) are classical but they do not 
have the right slopes. Indeed, as mentioned above (see Sec.~\ref{serexmp}), the slope of the form obtained in (b) is not necessarily $\alpha+1$ whereas
the form in (c) has infinite slope. However, in this paper we develop a method to construct a classical form $g$ 
of finite slope $\alpha + 1$ which is closely related to all three forms above in the sense that $g \equiv \theta f \mod p$. 
In fact, more generally,  given an eigenform $f$ of slope $\alpha$ as in \thmref{main1}, we prove  in Sec.~\ref{section main} that 
there exists an eigenform $g$ of 
slope $\alpha +\kappa-1$ such that
\begin{eqnarray}
  \label{g cong theta f}
g \equiv \theta^{\kappa-1} f \mod p
\end{eqnarray}
(see \thmref{main22}). Theorem~\ref{main1} and Corollary~\ref{main} now follow immediately from \eqref{g cong theta f}.
The proof of \eqref{g cong theta f} uses families of overconvergent eigenforms and congruences between them. We spend some time
recalling the necessary background about such families in Sec.~\ref{Colemanfamilies}. Finally, we remark that
while the forms $f$ and $f_k$ satisfying the congruence conditions on the weight $k$ in Theorem~\ref{main1} and \corref{main} may not be
typical, examples of such forms are not hard to write down under some plausible
assumptions on the size of $M_{f_k}$.

One of the chief motivations of this paper was to develop and use a result like \corref{main} to study the images of the reductions of local
Galois representations associated with eigenforms of arbitrary weights and slopes.
Assume $p \ge 5$. Given a Galois representation $\bar{\rho}_f:{\rm Gal}(\bar{\mathbb{Q}}/\mathbb{Q}) \rightarrow {\rm GL}_2(\overline{\mathbb{F}}_p)$ coming from an eigenform $f$, we can restrict it to the subgroup $G _p:={\rm Gal}(\bar{\mathbb{Q}}_p/\mathbb{Q}_p)$ to get a local representation $\bar{\rho}_{f,p}:=\bar{\rho}_f|_{G_p}.$ This representation has been much studied, but its shape is not known in general. However, if $f$ and $g$ are as in Corollary~\ref{main}, 
then clearly $\bar{\rho}_{f,p} $ is irreducible (respectively, reducible) if and only if $\bar{\rho}_{g,p}$ is irreducible (respectively, reducible). So we obtain the following.
\begin{cor}
  \label{irred vs red}
  The phenomenon of irreducibility (respectively, reducibility) of the reduction of local modular Galois representations tends to propagate as the slope
  increases by one.
\end{cor}

To say more, we work in the more general setting of crystalline representations of $G_p$. Let $I_p \subset G_p$ denote the inertia subgroup at $p$.
Let $\omega_2$ denote the mod $p$ fundamental character of level $2$ of $G_{p^2} = {\rm Gal}(\bar{\mathbb{Q}}_p/\mathbb{Q}_{p^2}) \subset G_p$. 
If $t$ is an integer with $p+1 \nmid t$, let ${\rm ind}(\omega_2^t)$ be the unique irreducible 
two-dimensional mod $p$ representation of $G_p$, with determinant $\omega^t$ and with restriction to $I_p$ given by $\omega_2^{t} \oplus \omega_2^{pt}$.
Let $E$ be a finite extension of $\mathbb{Q}_p$ and $a_p \in \mathfrak{m}_E$, the maximal ideal in the ring of integers $\mathcal{O}_E$. 
Let $v$ denote the normalized $p$-adic valuation so that $v(p) = 1$. For $k\ge 2$, let $V_{k,a_p}$ be the irreducible 
crystalline representation of $G_p$ defined over $E$, with Hodge-Tate weights $(0, k - 1)$ and slope $v(a_p)$, 
such that $D_{\rm cris}(V^*_{k,a_p}) = D_{k,a_p}$, where $D_{k,a_p} = Ee_1 \oplus Ee_2$ is the filtered $\phi$-module defined in \cite{ber}. 
Let $\bar{V}_{k,a_p}^{ss}$ denote the semisimplification of the mod $p$ reduction of 
any $G_p$-stable $\mathcal{O}_E$-lattice in $V_{k,a_p}$; it is independent of the choice of the lattice. 
For any normalized eigenform $f=\sum_{n= 1}^\infty a_nq^n \in S_k(N)$, with $k \ge 2$, $(p,N)=1$, $\chi = 1$ and slope $v(a_p) > 0$, it is known that
$$
\bar{\rho}_{f,p} \simeq \bar{V}_{k,a_p}^{ss}.
$$

For a fixed $k$ and $a_p$, there are only finitely many possibilities for $\bar{V}_{k,a_p}^{ss}$, up to unramified characters. But computations show that the behavior of this reduction is quite mysterious. The shape of $\bar{V}_{k,a_p}^{ss}$ is known when $k$ is small $(k \le 2p + 1)$, by the work of Fontaine, Edixhoven and Breuil \cite{edi}, \cite{bre}. On the other hand, Berger, Li and Zhu \cite{blz} computed its shape when the slope is large compared to $k$, that is, when $v(a_p) > \lfloor  \frac{k-2}{p-1}\rfloor$. Deligne  \cite{del} obtained the shape of $\bar\rho_{f,p}$ in the case $v(a_p)=0$. Recently, Buzzard and Gee \cite{buzgee1, buzgee2} computed the reduction $\bar{V}_{k,a_p}^{ss}$, when $v(a_p) \in (0,1)$ and  Bhattacharya, Ganguli, Ghate, Rai and Rozensztajn \cite{GG15, bg15, bgr18, gr19} computed the reduction if $v(a_p) \in [1, 2)$, for all weights $k \ge 2$.

 In the remainder of this section, 
we assume $f$ and $g$ to be eigenforms as in \corref{main} of
slope $\alpha$ and $\alpha +1$, respectively, so $\bar{\rho}_f \otimes \omega \simeq \bar{\rho}_g$. We also assume $\chi = 1$ for simplicity.
Let $\alpha \in [0,1)$. Since the structure of $\bar{\rho}_{f,p}$ is known due to Deligne, Buzzard and Gee,
we immediately get the structure of $\bar{\rho}_{g,p}$, by \corref{main}. 
On the other hand, the structure of $\bar{\rho}_{g,p}$ has been independently and directly computed by the authors above.  
Hence, we can compare the two results regarding the structure of $\bar{\rho}_{g,p}$, 
one derived from the structure of the smaller slope eigenform $f$ and Corollary~\ref{main}, and the other directly from results 
in the literature. We do this in Sec.~\ref{comp}. 
In all cases (see Tables 1-3 in Sec. \ref{slopes in [0,1)}), the two methods to compute the structure of $\bar\rho_{g,p}$ are compatible (as they should be!).
We can use \corref{main} to partially predict the structure of $\bar{\rho}_{g,p}$ for forms $g$ of slope in $[2,\infty)$, a range of slopes 
for which the shape of the reductions have yet to be determined completely. We illustrate this in Sec. \ref{extrap} for slope in $[2,3)$.  
Recently, the first author made a general conjecture, known as the zig-zag conjecture, describing the 
reductions of local Galois representations associated with cusp forms of positive half-integral slopes $\le \frac{p-1}{2}$ and exceptional weights, 
for which computing the reduction is the trickiest (see \cite{gha19}). We show that \corref{main} is compatible with the zig-zag conjecture 
in the cases where the reduction is known (cf. Secs. \ref{slope 0} and \ref{slope 1/2}).  In \cite{gha19}, it was shown
that \corref{main} is compatible with the zig-zag conjecture in general (cf. Sec. \ref{zigzag}). 
 
Coleman families are not defined on all of weight space, whence the notion of the Coleman radius $p^{-r}$ of a Coleman family 
of finite slope $\alpha$, where $r$ is a rational number in the range $(-1+\frac{1}{p-1}, \infty)$. Gouv\^ea-Mazur \cite{gm92}
gave a conjectural upper bound $\lceil \alpha \rceil$ for a quantity closely related to $r$, 
which was subsequently shown to be not 
valid in all cases by Buzzard-Calegari \cite{bc}. Wan \cite{wan} showed that this quantity is bounded above by a quadratic polynomial 
in $\alpha$. This should provide lower bounds for the largest Coleman radius $p^{-r}$ of a Coleman family 
passing through a form of finite slope $\alpha$.  
Since the radius of the Coleman family passing through the cusp form $g$ in \corref{main} is involved in defining its weight,
as an application of the above-mentioned compatibility, in Sec. \ref{lb} we obtain an {\it upper bound} for the radius of the Coleman 
family passing through $g$, when 
$g$ has slope in $[1,2)$. 

Let $f$ be an eigenform of finite slope $\alpha$ such that $\bar\rho_f$ is irreducible.  
We end this paper in Sec.~\ref{serexmp} by giving examples of cases where the minimal weight eigenform 
$h$ associated to $\bar{\rho}_f \otimes \omega$ by Serre's conjecture does not have slope $\alpha + 1$. This shows 
that the constructions made to prove Theorem~\ref{main1} and Corollary~\ref{main} are of some importance.


\section{Overconvergent modular forms}
 \label{Colemanfamilies}

In this section, we recall some basic results on families of overconvergent $p$-adic modular forms that will be needed in our proof.
%
The definition of overconvergent modular forms was first given by Katz (see, e.g., \cite[Sec. 2]{wan}). 
Katz's overconvergent modular forms are  defined only for integral weights. 
In \cite{col97}, Coleman defined overconvergent modular forms of tame level $N$ of 
weights $\kappa \in \mathcal{W}:= {\rm Hom}(\mathbb{Z}_{p}^*,\mathbb{C}_{p}^*)$, the $p$-adic weight space, which is a $p$-adic rigid analytic space. 
These forms incorporate the forms of Katz, 
since we have an embedding $\mathbb{Z} \hookrightarrow  \mathcal{W}$ sending $k\in \mathbb{Z}$ to the character $a\mapsto a^k$, for $a\in \mathbb{Z}_p^*$. 
Coleman's definition of overconvergent modular forms of integral weights is geometric, while forms of general weight in $\mathcal{W}$ 
are defined using powers of a weight $1$ Eisenstein series.  
We recall the definition of overconvergent forms  following Coleman \cite{col97}.  
The reader is referred to \cite{col96, col97} for a more systematic treatment, and \cite{BGR} for background on rigid analysis.

Let $N \ge 1$ be the tame level and $p\ge 5$ a prime that is relatively prime to $N$.  
Let $I_m:= \{v \in \mathbb{Q}: 0\le  v < \frac{p^{2-m}}{p+1}\}$ and $I_m^*=I_m \backslash \{0\}$, for an integer
$m \ge 1$. 
Let $A$ be a lift of the Hasse invariant to the modular curve $X_1(N)$ (since $p \ge 5$, we can take $A=E_{p-1}$, the Eisenstein series of level 1 and weight $p-1$). 
Then, for $v \in I_1$,
$$
X_1(N)(v):= \{x\in X_1(N): v(A(x)) \le v \}
$$
is an affinoid subdomain of $X_1(N)$. Using the canonical subgroup, we may regard $X_1(N)(v)$ as an affinoid subdomain of  $X_1(N;p)=X(\Gamma_1(N)\cap \Gamma_0(p))$, 
by \cite[Sec. 6]{col96}. Let $X_1(Np)(v)$ be the affinoid subdomain of $X_1(Np)$ which is the inverse image of $X_1(N)(v)$ under the natural forgetful map from 
$X_1(Np)$ to $X_1(N;p)$. Let $\omega$ denote the invertible sheaf on $X_1(Np)$ defined as in \cite[Sec. 2, 8]{col96}. On the non-cuspidal locus, $\omega$ 
is the push-forward of the sheaf of relative invariant differentials of the universal elliptic curve. Then, for $k\in \mathbb{Z}$,
$$
M_{k}^\dagger(v):=\omega^k(X_1(Np)(v))
$$
is the space of $v$-overconvergent modular forms of weight $k$ on $\Gamma_1(Np)$ which converge on $X_1(Np)(v)$ \cite[p. 449]{col97}. 
If $0 < v' < v$, then there is an injection $M_k^\dagger(v)\hookrightarrow M_k^\dagger(v')$ and their direct limit
$$
M_k^\dagger(N)=\varinjlim_{v>0}M_k^\dagger(v)
$$
is the space of overconvergent modular forms of weight $k$ and level $Np$. Let $S_k^\dagger(N)$ denote the subspace of cusp forms in $M_k^\dagger(N)$ (the subspace of functions vanishing at the cusps in $X_1(Np)(0)$). 
Every  overconvergent modular form has a $q$-expansion and the $v$-adic valuation of the $p$-th Fourier coefficient of a (normalized) form is called the 
{\it slope} of the form.

Note that $M_k^\dagger(N)$ is an infinite-dimensional $p$-adic Banach space and contains the classical modular forms of weight $k$ and level $Np$. 
Using the Hodge theory of modular curves, Coleman \cite[Theorem 8.1]{col96} proved the following `control theorem'.
 
\begin{thm}\label{classicity}
  An overconvergent modular form of weight $k$ and tame level $N$ 
  is a classical modular form if its slope is strictly less than $k-1$. 
\end{thm}


Let $\theta = q \frac{d}{dq}$ be the theta operator whose action on $q$-expansions is given by  \eqref{theta}.
We recall that a suitable power of $\theta$ preserves overconvergent forms (see \cite[Proposition 4.3]{col96} and \cite{colhl}).

\begin{thm}\label{thetaover}
Let $\kappa \ge 2$ be an integer and $f$ be an overconvergent modular form of weight $2-\kappa$ and tame level $N$ and some character. 
Then $\theta^{\kappa-1}f$ is also an overconvergent form of weight $\kappa$, tame level $N$ and the same character.  Moreover, $\theta^{\kappa-1}f$
is an eigenform if $f$ is an eigenform, and is $N$-new if $f$ is $N$-new. 
\end{thm}

\subsection*{Notation} Let $|\cdot|$ be the norm on ${\mathbb C}_p$, defined by $|x| = p^{-v(x)}$. 
Let $\pi$ be a $(p-1)$-st root of $-p$, so $|\pi/p| = p^{\frac{p-2}{p-1}} > 1$, and let $\mathfrak{B}^*=B_{\mathbb{Q}_p}(0,|\pi/p|)$ 
be the extended disc. Let $D=\left(\mathbb{Z}/p\mathbb{Z}\right)^*$ and $\hat{D} = \mathrm{Hom}(D, \mathbb{C}_p^*) = \mathbb{Z}/(p-1) \mathbb{Z}$. 
The rigid analytic space $\mathcal{W}^* = \hat{D} \times \mathfrak{B}^*$ sits inside $\mathcal{W}$, where we identify 
a point $(i,s) \in \mathcal{W}^*$ with the character $a\mapsto \omega^i(a)  \langle\langle a\rangle\rangle^s $, for $a \in {\mathbb {Z}}_p^*$. 
Here, $\omega$ is the Teichm\"uller character on $\mathbb{Z}_p^*$ and $\langle\langle \cdot \rangle\rangle: \mathbb{Z}_p^* \longrightarrow \mathbb{Z}_p^*$ 
is the character defined by $\langle\langle a \rangle\rangle=\frac{a}{\omega(a)}$, so $\langle\langle a \rangle\rangle \equiv 1 \mod p$, 
for all $a\in \mathbb{Z}_p^*$. 

\subsection*{Construction of the characteristic power series of the $U_p$ operator}

For $v\in I_1^*$, the $U_p$ operator (sometimes we write $U_{(k)}$) is a completely continuous endomorphism of $M_k^\dagger(v)$. 
The characteristic power series of the $U_p$ operator plays a major role in the theory of overconvergent modular forms. 
We give some details.

For simplicity, we write $X(v)$ for $X_1(Np)(v)$, for $v \in I_1$. Let $E$ be the weight one modular form on $\Gamma_1(p)$ 
with character $\omega^{-1}$ as defined in \cite[p. 447, (1)]{col97}. One has $E \equiv 1 \mod p$.  Let 
$e$ be the analytic function on $\bigcup_{v\in I_2}X(v)$ with $q$-expansion $E(q)/E(q^p)$, so  $|e-1|_{X(0)}\le |p|$. 
The following fact is standard, see \cite[Proposition 5.8]{was}: for $s \in \mathbb{C}_p$, the power series $\sum_{n=0}^\infty \binom{s}{n}T^n$ converges for $|T| < |p|r$ if  $|s| \le |\pi/p|(1/r)$, for $r \le |\pi/p|$. Taking $r = 1$, we obtain
%

$$e^s := \sum_{n=0}^\infty \binom{s}{n} (e-1)^n$$ is a well-defined function on $X(0)$,
for $s\in \mathbb{C}_p$ with $|s| \le|\pi/p|$. We need a variant of this.
As $|e-1|_{X(0)}=\lim_{v\rightarrow 0^{+}}|e-1|_{X(v)}$, we have the following lemma.

\begin{lem}\label{vexist} {\rm (cf.} \cite[Lemma B3.1]{col97}{\rm )}
 For any $\epsilon \in \mathbb{R}$ with $|p|<\epsilon$, there exists a $v\in I_2^*$ such that $e$ is defined on $X(v)$ and $|e-1|_{X(v)}< \epsilon$.
\end{lem}

We obtain the following.

\begin{lem}\label{conv}
If $t \in [1, |\pi/p|)$, then there exists a $v\in I_2^*$ such that  $|e-1|_{X(v)} < |\pi|/t$. 
The function $e^s$ is defined on $X(v)$, for all $s$ such that $|s| \le t$. 
\end{lem}
\begin{proof}
By  \lemref{vexist} with $\epsilon = |\pi|/t$, 
there exists a $v\in I_2^*$ such that $|e-1|_{X(v)} < |\pi|/t$. 
From the standard fact above applied with $r = |\pi/p|(1/t) \le |\pi/p|$, we obtain that $e^s$ is defined on $X(v)$, for all $s$ with $|s| \le t$.
\end{proof}


For all $s \in \mathfrak{B}^*$ with $|s|\ge 1$ and $v \in I_2^*$ with $|e-1|_{X(v)} < |\pi/s|$, we see that $e^s$ is defined on $X(v)$, so the operator 
$u_s := U_{(0)}\circ m_{e^s}$ on $M_0^\dagger(v)$ is well defined, where $m_{e^s}$ is 
`multiplication by the function $e^s$'. Clearly $u_s$ is a completely continuous operator because  $U_{(0)}$ is.
Let $k\in \mathbb{Z}$ and $v\in I_2^*$ be such that $e^k$ is defined on $X(v)$. Then Coleman observed \cite[p. 451, (1)]{col97} 
that the Fredholm theory of the operator $U_{(k)}$ on $M_k^\dagger(v)$ is equivalent to that of $u_k$ on $M_0^\dagger(v)$, i.e., we have an equality of Fredholm determinants
\begin{equation}\label{id0k}
\det(1-TU_{(k)}|M_k^\dagger(v))= \det(1-Tu_k|M_0^\dagger(v)).
\end{equation}

Coleman interpolates these Fredholm determinants by constructing a power series $P(s,T)$, for $s\in \mathfrak{B}^*$, which we describe now. 
Let
\begin{eqnarray*}
\mathscr{T}^*:=\{ (t,v)  \in (|\mathbb{C}_p|\cap[1,|\pi/p|)) \times I_2^* : |e-1|_{X(v)}< |\pi|/t  \}. 
\end{eqnarray*}
The set $\mathscr{T}^*$ is non-empty, by Lemma~\ref{conv}.
Let $K$ be a finite extension of $\mathbb{Q}_p$.
Put $\mathcal{V}^*:= \bigcup_{(t,v)\in \mathscr{T}^*}Z_t(v)$, which is a rigid analytic subspace of $\mathbb{A}^1_{/K} \times X_1(Np)_{/K}$ admissibly covered by the affinoids 
$$Z_t(v)=B_K[0,t]\times_K X(v)_{/K}.$$  Write $A(X)$ for the algebra of rigid analytic functions on a rigid analytic space $X$. Let
$$
M(t,v):= A(Z_t(v)) = A(B_K[0,t]) \> \hat{\otimes}_K \> M_0^\dagger(v)
$$
be the algebra of rigid analytic functions on $Z_t(v)$.

Let $(t,v) \in \mathscr{T}^*$ and  $s \in B_K[0,t]$.  The projection $$Z_t(v) \rightarrow B_K[0,t]$$ 
makes $M(t,v)$ into an $A(B_K[0,t])$-module.
We may view $u_s$ as a completely continuous operator $u_s : A(Z_t(v)_s)\rightarrow A(Z_t(v)_s)$, where  $Z_t(v)_s=\{s\} \times X(v)$ is the fiber above $s$.  
Coleman shows there is a completely continuous operator $U_{(t,v)}$  on $M(t,v)$ over $A(B_K[0,t])$ 
whose restriction to the fiber $Z_t(v)_s$ above $s$ 
is $u_s$, i.e., the following diagram commutes:
\begin{equation}
  \label{U vs u}
\begin{tikzcd}
A(Z_t(v)) \arrow[r] \arrow[d, "U_{(t,v)}"]
& A(Z_t(v)_s) \arrow[d, "u_s" ] \\
A(Z_t(v)) \arrow[r]
& A(Z_t(v)_s).
\end{tikzcd}
\end{equation}
Since $A(B_K[0,t])$ is a PID, using \cite[Lemma A5.1]{col97} we see that $M(t,v)$ is orthonormizable over $A(B_K[0,t])$. 
Thus, for any $(t,v) \in \mathscr{T}^*$, we may define the characteristic power series 
$$
P_{(t,v)}(s, T):= \det (1-TU_{(t,v)}|M(t,v)) \in A(B_K[0,t])[[T]].
$$
The series $P_{(t,v)}$ is analytic on $B_K[0,t] \times \mathbb{C}_p$, and
is independent of $(t,v)$, in the sense that if $(t,v)$ and $(t',v')$ lie in $\mathscr{T}^*$ with $t\le t'$, 
then the restriction of $P_{(t',v')}$ from  $B_K[0,t'] \times \mathbb{C}_p$ to  $B_K[0,t] \times \mathbb{C}_p$ is $P_{(t,v)}$.
Using \eqref{id0k} and \eqref{U vs u},  
we have the following result.
\begin{thm}{\rm (}\cite[Theorem B3.2]{col97}{\rm )}
There is a unique rigid analytic function $P(s,T)$ on $\mathfrak{B}^* \times \mathbb{C}_p$ defined over $\mathbb{Q}_p$, such that for $k\in \mathbb{Z}$ and $v\in I_1^*$, 
$$
P(k,T)=\det(1-TU_{(k)}|M_k^\dagger(v)).
$$
\end{thm}
Note that $D$ acts via the diamond operators on all the spaces $M_k^\dagger(v)$ and $M(t,v)$. From the decomposition of the space 
$M(t,v)=\oplus_{\varepsilon \in \widehat{D}}M(t,v,\varepsilon)$, we get
$$
P(s,T)=\prod_{\varepsilon \in \widehat{D}}P_\varepsilon(s,T). 
$$
Similarly, the decomposition $S(t,v) = \oplus_{\varepsilon \in \widehat{D}} S(t, v,\varepsilon)$ for cusp forms induces 
$$
P^0(s,T)=\prod_{\varepsilon \in \widehat{D}}P^0_\varepsilon(s,T). 
$$
If $i$ denotes an integer such that $0\le  i < p-1$, we set $P_{i}(s, T)=P_{\omega^i}(s, T)$ and
$P_{i}^0(s, T)=P_{\omega^i}^0(s, T)$.
We obtain the following.
\begin{thm}\label{chara2}{\rm (}\cite[Theorem B3.3]{col97}{\rm )}
For each $0 \le i < p-1$, there exists a rigid analytic function $P_{i}(s, T) \in \mathbb{Q}_p[[s,T]]$ which converges on the region $ \mathfrak{B}^*\times \mathbb{C}_p$ 
such that for an integer $k$, $P_{i}(k, T)$ is the characteristic series of the $U_p$ operator acting on overconvergent forms of 
weight $k$ and character $\omega^{i-k}$. 
\end{thm}
\noindent Similarly, for cusp forms,  we have 
$$
P_i^0(k, T) = \det(1 - TU_{(k)}|S_k^\dagger(N, \omega^{i-k})).
$$

Let $d(k,\varepsilon,\alpha)$ (respectively, $d^0(k,\varepsilon,\alpha)$) 
 be the dimension of the subspace of classical modular forms (respectively, classical cusp forms) of weight $k$ and character $\varepsilon \omega^{-k}$ consisting of forms 
of finite slope $\alpha$,
for $\varepsilon\in\widehat{D}$. Then from \thmref{chara2}, the series $P_{\varepsilon}(s, T)$ is continuous in the $s$-variable, so for integers $k$, $k'$
which  are  sufficiently close $p$-adically,
the series $P_{\varepsilon}(k, T)$ and $P_{\varepsilon}(k', T)$ are $p$-adically close (see \cite[Theorem 1]{gm93}), 
so their Newton polygons are close and therefore equal, so the number of zeros of $P_\varepsilon(k,T)$ and 
$P_\varepsilon(k',T)$ of valuation $-\alpha$ in the $T$-variable are the same. If $k$, $k' > \alpha +1$, the corresponding forms are classical by
the control theorem (Theorem~\ref{classicity}). 
A similar argument works for  $P^0_{\varepsilon}(s, T)$. We obtain the following.

\begin{thm} {\rm (}\cite[Theorem B3.4]{col97}{\rm )}
  If  $\varepsilon\in\widehat{D}$, $\alpha \in \mathbb{Q}$ and $k$, $k' > \alpha + 1$ are integers which are sufficiently close $p$-adically, then 
  \begin{eqnarray*}
    d(k,\varepsilon,\alpha) & = & d(k',\varepsilon,\alpha), \\
    d^0(k,\varepsilon,\alpha) & = & d^0(k',\varepsilon,\alpha).
  \end{eqnarray*}   
\end{thm}

This hints at the existence of $p$-adic families of overconvergent modular forms to which we turn now. 

\subsection*{Coleman families}
Coleman defined a Hecke algebra which acts on the space of families of overconvergent modular forms and used it, together with Riesz 
theory over affinoids, to prove that any overconvergent modular form lives in a family. We describe this result in detail. 
For a similar result in the more general Hilbert modular setting, see \cite[Theorems 3.16 and 3.23]{aip}. 

We first describe the notion of a Coleman family. For each $0 \le i < p-1$, let 
$$
M^\dagger(N,i)=\varprojlim_{t\le |\pi/p|}\varinjlim_{(t,v)\in \mathscr{T}^*}M(t,v,\omega^i)
$$
and
$$
S^\dagger(N,i)=\varprojlim_{t\le |\pi/p|}\varinjlim_{(t,v)\in \mathscr{T}^*}S (t,v,\omega^i),
$$
and let $M^\dagger(N) = \oplus_i M^\dagger(N,i)$, $S^\dagger(N) = \oplus_i S^\dagger(N,i)$.
These are $A(\mathfrak{B}^*)$-modules.

Suppose $\alpha$ is a positive rational and $k_0 \in \mathfrak{B}^*(K)$ is an integer.
Coleman showed that there is an integer $d \ge 0$
such that there exists an affinoid disk $B=B_K[k_0,p^{-r}]$, with $0<p^{-r}<|\pi/p|$ (so $r \in (-1 + \frac{1}{p-1}, \infty))$, such 
that the slope $-\alpha$ affinoid in the zero locus of $P_{i}^{0}(s,T)$, 
for $0 \le i < p-1$, is 
finite of degree $d$ over $B$. The integer $d$ equals the dimension 
of $S_{k_0}^{\dagger}(N,\omega^{i-k_0})_\alpha$,  
the subspace of  $S_{k_0}^{\dagger}(N,\omega^{i-k_0})$ 
consisting of forms of slope $\alpha$. 
If $Q$ is the corresponding factor of $P_{i}^{0}(s,T)$ 
over $B$, then by the Riesz decomposition theorem \cite[Theorem A4.3]{col97} for the module 
$S^{\dagger}(N,i)_B := S^{\dagger}(N,i) \otimes_{A(\mathfrak{B}^*)} A(B)$, 
we get a free closed submodule 
$H:=N_{U_B}(Q)$ of rank $d$ over $A(B)$ such that  $Q^*(U_B)H^{}=0$, 
where $U_B$ is the restriction of $U_p$ on $S^{\dagger}(N,i)_B$ 
and for a polynomial $Q(T)$ of degree $d$, we define $Q^*(T)=T^dQ(T^{-1})$. Let $\mathbb{T}$ denote the Hecke algebra generated over $A(\mathfrak{B}^*)$ by
the Hecke operators $T(n)$. Let $R$ 
be the image of $\mathbb{T}\otimes_{A(\mathfrak{B}^*)} A(B)$ in ${\rm End}_{A(B)}(H)$, so $R$ is free of finite rank $d$ over $A(B)$. 
Moreover, $R$ is the ring of rigid analytic functions on an affinoid $X(R)$ with a finite morphism to $B$ of degree $d$.
We obtain the following theorem (this is \cite[Theorem B5.7]{col97}, for weights larger than $\alpha + 1$ in $B$).

\begin{thm}
Suppose $L \subset \mathbb{C}_p$ is a finite extension of $K$. For $x\in X(R)(L)$, let $\eta_x: R \rightarrow L$ be the corresponding homomorphism and set
$$
f_x(q) = \sum_{n=1}^\infty \eta_x(T(n))q^n.
$$
Now suppose $k$ is an integer such that $k\in B$. Then the mapping from $X(R)_{k}(L)$ to $L[[q]]$, $x\in X(R)_{k}(L) \mapsto f_x(q)$, 
is a bijection onto the set of $q$-expansions of normalized overconvergent cuspidal eigenforms on $X_1(Np)$ over $L$  of weight $k$, character $\omega^{i-k}$ and slope $\alpha$.
Moreover, if $k > \alpha + 1$, the bijection is onto the corresponding space of classical forms. 
\end{thm}

We need a slight refinement of the above result.  We have a surjection
\begin{eqnarray*}
  R \otimes_{A(B)} \mathrm{QF}(A(B)) \twoheadrightarrow \prod_\lambda {K_\lambda},
\end{eqnarray*}
where $K_\lambda$ are finite extensions of the quotient field $\mathrm{QF}(A(B))$ of $A(B)$ and the $\lambda : R \rightarrow K_\lambda$ vary over the (finitely many) $\overline{\mathrm{QF}(A(B))}$-valued points of $R$.  
The above surjection is an isomorphism if $R$ is semi-simple, but the semi-simplicity of the $U_p$-operator does not seem to be known. 
Let $A_\lambda$ be the integral closure of $A(B)$ in $K_\lambda$. It is also an affinoid, being a finitely generated $A(B)$-module. We let
 $\mathcal{U}_\lambda = X(A_\lambda)$ be the corresponding
rigid analytic space. There is a morphism $\mathcal{U}_\lambda \rightarrow B$ which is of finite degree $d_\lambda \ge 0$, and $\sum_\lambda {d_\lambda} = d$
if $R$ is semi-simple. 

\begin{defn}
  The $q$-expansion $\mathcal{F_\lambda} =  \sum_{n=1}^\infty \lambda(T(n)) q^n \in A_\lambda[[q]]$ is called a Coleman family. 
\end{defn} 
Note that $a_n := \lambda(T(n)) \in A_\lambda$, so the $a_n$ may be thought of as rigid analytic functions on $\mathcal{U}_\lambda$ and so 
may be evaluated at points in $\mathcal{U}_\lambda$ lying over integer points in $B$. We say that an overconvergent
cuspidal eigenform $f$ of integral weight $k_0$ and character $\omega^{i-k_0}$ lives in the Coleman family $F_\lambda$ 
if there is a homomorphism $\eta_y : A_\lambda  \rightarrow L$, corresponding to $y \in \mathcal{U}_\lambda$, for $L$ the field of definition of $f$, such that 
$\eta_{y}(F_\lambda) = \sum_{n=1}^\infty \eta_{y}(a_n) q^n = \sum_{n=1}^\infty a_n(y) q^n \in L[[q]]$ is the $q$-expansion of $f$.  

\begin{thm}
  Every normalized overconvergent cuspidal eigenform $f$ of integral weight $k_0$ and character $\omega^{i-k_0}$ lives in a Coleman family (of type $i$). 
\end{thm}

\begin{proof}
  Let $L$ be the field of definition of $f$. 
  By the previous theorem, there is an affinoid ball $B = B_K[k_0,p^{-r}]$, for some $r$ as above, and a morphism $\eta_x: R \rightarrow L$
  such that $f_x = f$. Let $\mathfrak{m}_x = \ker \eta_x$, a maximal ideal of $R$. Take a minimal prime $\mathfrak{p}$ of $R$ such that
  $\mathfrak{p} \subset \mathfrak{m}_x$. Then $R/\mathfrak{p}$ is a domain and its quotient field is a finite integral extension of $\mathrm{QF}(A(B))$. 
  Consider the homomorphism $\lambda : R \rightarrow R/\mathfrak{p} \subset A_\lambda$. We claim that $f$ lies in the Coleman family $\mathcal{F}_\lambda$.
  But this is obvious in view of the fact that $\eta_x$ is the composition of the maps $\lambda$ and (any extension $\eta_y$ to $A_\lambda$ of) the
  canonical projection $R/\mathfrak{p} \rightarrow R/\mathfrak{m}_x = L$.
\end{proof}

Everything we have said above holds for the subspaces of $N$-new forms, as is easy to see. Coleman calls these forms $p'$-new  
on \cite[Definition, p. 467]{col97}. Summarizing the above discussion, and including the case of newforms, and the case 
$\alpha = 0$ which was known earlier by work of Hida \cite{hi86}, \cite{hi86b}, and introducing an auxiliary character $\chi$ of level $N$, we obtain the following theorem:

\begin{thm} \label{familyU} Let $p \ge 5$, $p \nmid N \ge 1$, $\chi$ a character of level $N$ and $k_0\in \mathbb{Z}$. 
Suppose that $f$ is a normalized overconvergent cuspidal eigenform of weight $k_0$, level $Np$, character $\chi \omega^{i-k_0}$ and slope $\alpha \ge 0$, defined over 
a finite extension $L$ of $K$. 
Then there exists an $r \in \mathbb Q \cap (-1+\frac{1}{p-1}, \infty)$ and a rigid analytic space $\mathcal{U} \rightarrow  B_K[k_0,p^{-r}]$  
and a family $\mathcal{F}=\sum_{n = 1}^\infty a_nq^n$, where $a_n$'s are rigid analytic functions on $\mathcal{U}$ such that the following statements hold:
\begin{enumerate}
 \item For every integer $k \in B_K[k_0,p^{-r}]$, and every point $\eta_y \in \mathcal{U}$ lying over $k$, the series $f_k:=\sum_{n = 1}^\infty \eta_y(a_n)q^n$ coincides with the $q$-expansion of a 
        normalized overconvergent cuspidal eigenform of weight $k$, level $Np$, character $\chi \omega^{i-k}$ and slope $\alpha$. Moreover, if $f$ is $N$-new, then all the $f_k$ 
        are $N$-new.
\item  There is an $\eta_{y_0} \in \mathcal{U}$ lying over $k_0$ such that the series  $\sum_{n =  1}^\infty \eta_{y_0}(a_n)q^n$ coincides with the $q$-expansion of $f$.
\end{enumerate}
\end{thm}

\begin{defn}\label{radius}
We call the radius $p^{-r}$ appearing in  \thmref{familyU} the Coleman radius of the family $\mathcal{F}$.
Let $r_f$ be such that $p^{-r_f}$ is the supremum of the numbers $p^{-r}$ for which \thmref{familyU} holds.
Let $M_f \in \mathbb{Z}_{\ge 0}$ be the smallest non-negative integer such that $p^{-{M_f}}$ is a Coleman radius.  
\end{defn}
\noindent We remark that if $r_f \in {\mathbb Z}$, then $M_f = r_f$ if and only if $p^{-r_f}$ is a Coleman radius.
If the slope  $\alpha =0$, we know that\footnote{There is an isomorphism 
$B(0, |\pi/p|) \xrightarrow{\sim} B(0,1)$ induced by $s \mapsto (1+p)^s-1$. Coleman works on the left hand side of this 
isomorphism whereas Hida works on the right, so that when $\alpha = 0$, we have 
$r_f = -1 + \frac{1}{p-1}$ and $M_f = 0$.} $M_f = 0$ by the work of Hida \cite{hi86}, \cite{hi86b}.


\subsection*{Congruences}

We now show that the forms $f_k$ in a Coleman family are all congruent modulo $p$, 
at least if one takes $k$ in the interior of the ball $B_K[k_0,p^{-M_f}]$ in the case when $M_{f} = r_f$. 
To this end, let $\delta_f = \delta_{M_f,r_f}$ denote the Kronecker delta function defined by
\begin{eqnarray}
  \label{delta}
\delta_f & = & \begin{cases}
                   1, \quad {\rm if~} M_{f} = r_{f},\\
                   0,  \quad {\rm otherwise}.
                \end{cases}
\end{eqnarray}
If the slope $\alpha =0$, we have $\delta_f = 0$. 

\begin{prop}
 Let the notation be as in Theorem~\ref{familyU} and $M_f$ 
be as in Definition~\ref{radius}.
 For  every integer $k \in B_K[k_0,p^{-M_f}]$, with $k\equiv k_0 \mod
p^{M_f+\delta_f}$, we have
\begin{equation}\label{familycong}
  f_k \equiv f \mod p.
\end{equation}
\end{prop}

\begin{proof}
Choose $r$ such that $r_f  < r < M_f + \delta_f$, and let $B =  B_K[k_0,p^{-r}]$.
The Hecke operators in $R$ have norm bounded by $1$, so the $a_n$ also have norm bounded by $1$.
In particular, the $a_n$ lie in the subring of power bounded elements $A_\lambda^0$ of $A_\lambda$.
Let $\mathcal{O}_K$ be the ring of integers of $K$.
Note that $A_\lambda^0$ contains $A^0(B) = \mathcal{O}_K \langle (s-k_0)/p^{r} \rangle$, the power bounded 
elements in $A(B)$, and that  $A_\lambda^0$ is the integral closure of  $A^0(B)$ in $K_\lambda$. 
Let $\widetilde{A}(B)$ be the completion of $A(B)$ in the $s'$-adic topology with $s' = (s-k_0)/p^{r}$ and
$\widetilde{A}^0(B) = \mathcal{O}_K [[ (s-k_0)/p^{r} ]]$ be the subring of $\widetilde{A}(B)$ consisting of
power bounded elements. Let $\widetilde{K}_\lambda$ be the smallest field extension of ${\rm QF}(\widetilde{A}^0(B))$
containing $K_\lambda$, and $\widetilde{A}_\lambda^0$ the integral
closure of $\widetilde{A}^0(B)$ in $\widetilde{K}_\lambda$. 
Since $\widetilde{A}^0(B)$ is a complete local ring, we see that $\widetilde{A}_\lambda^0$ is a complete local ring.
Clearly, ${A}_\lambda^0 \subset \widetilde{A}_\lambda^0$. The above rings sit in the following diagram:
$$
\begin{tikzcd}[row sep=scriptsize, column sep=scriptsize]
& K_\lambda \arrow[from= dl, hook] \arrow[rr, hook] \arrow[from= dd, hook] & & \widetilde{K}_\lambda
\arrow[from= dl, hook] \arrow[from=dd, hook] \\
A_\lambda^0 \arrow[rr, hook, crossing over] 
& & \widetilde{A}_\lambda^0 \\
& {\rm QF}(A(B)) 
  \arrow[from=dl, hook] \arrow[rr, hook] & & {\rm QF}(\widetilde{A}^0(B)).
\arrow[from= dl, hook] \\
A^0(B) \arrow[rr, hook] \arrow[uu, hook]
& & 
\widetilde{A}^0(B)
\arrow[uu, hook, crossing over]\\
\end{tikzcd}
$$

Let $\eta_y$, $\eta_{y_0}$ be two homomorphisms $A_\lambda^0 \rightarrow
\mathcal{O}_L$ lying above $k$ and $k_0$ corresponding to $f_k$ and $f_{k_0}$, with kernels the prime ideals $\mathfrak{p}_y$, $\mathfrak{p}_{y_0}$, respectively.
Because of our choice of $k$ and $r$, the point $k$ lies in the interior of the ball $B$.  
So the restrictions of $\eta_y$, $\eta_{y_0}$ to $A^0(B) = \mathcal{O}_K \langle (s-k_0)/p^r \rangle$
extend to  homomorphisms
$$
\widetilde{A}^0(B) = \mathcal{O}_K [[ (s-k_0)/p^r ]] \rightarrow \mathcal{O}_L. 
$$
Since $\widetilde{A}_\lambda^0$ is a finite integral extension of $\widetilde{A}^0(B)$,  
we can extend the above homomorphisms to homomorphisms
$\widetilde{\eta}_y$, $\widetilde{\eta}_{y_0}:\widetilde{A}_\lambda^0 \rightarrow \mathcal{O}_{L'}$,
for some finite extension $L'$ of $L$. 
Let $\widetilde{\mathfrak{m}}_\lambda$ be the maximal ideal of the local
ring $\widetilde{A}^0_\lambda$. Let $\widetilde{\mathfrak{p}}_y$ and
$\widetilde{\mathfrak{p}}_{y_0}$ denote the prime ideals 
corresponding to the kernels of $\widetilde{\eta}_y$ and
$\widetilde{\eta}_{y_0}$. Clearly the two projections $\widetilde{A}^0_\lambda  \rightarrow
\widetilde{A}^0_\lambda /  \widetilde{\mathfrak{p}}_y
\subset \mathcal{O}_{L'}$ and
$\widetilde{A}^0_\lambda  \rightarrow \widetilde{A}^0_\lambda /
 \widetilde{\mathfrak{p}}_{y_0} \subset
\mathcal{O}_{L'}$ both further project to give the {\it same}
morphism
$\widetilde{A}^0_\lambda  \rightarrow \widetilde{A}^0_\lambda /
 \widetilde{\mathfrak{m}}_\lambda \subset
\mathcal{O}_{L'}/\mathfrak{m}_{L'}$. It follows that
the projections $A^0_\lambda  \rightarrow A^0_\lambda /  \mathfrak{p}_y 
  \subset \mathcal{O}_L$
  and  $A^0_\lambda  \rightarrow A^0_\lambda /  \mathfrak{p}_{y_0} \subset \mathcal{O}_L$ both further project to give the same morphism 
  $A^0_\lambda  \rightarrow \mathcal{O}_L/\mathfrak{m}_L$.
\end{proof}

We make some remarks. Firstly, the proposition implies that $\bar\rho_{f_k} \simeq \bar\rho_f$ up to semisimplification. 
Secondly, one might expect that the congruences \eqref{familycong} hold for all weights $k \in B_K[k_0,p^{-M_f}]$, as is the case when 
$\alpha = 0$ by Hida theory, but we have not been able to show this (but see \cite[Theorem D]{CM98}). 
However, when the map $\mathcal{U} \rightarrow B$ has degree 1, so that $\mathcal{U} = B$ is a ball, one can 
indeed show this under an additional assumption. In fact, one can prove the following stronger Kummer-like congruences
for the  members of the Coleman family, though they will not be used later.

The open ball $\mathfrak{B}^*=B_{K}(0,|\pi/p|)$ is isomorphic to the open unit ball $B^0=B_K(0,1)$. The map $s \mapsto (1+p)^s-1$ 
induces a morphism of rigid analytic functions $\mathcal{O}_K[[T]] \subset A(B^0) \rightarrow 
A(\mathfrak{B}^*)$ given by
$f(T)  \mapsto f((1+p)^s-1)$, for $f(T) \in \mathcal{O}_K[[T]]$.  
Recall that, with notation as above, we have $B = B_K[k_0,p^{-r}] \subset \mathfrak{B}^*$, for some $r$, 
so there is a map $A(\mathfrak{B}^*) \rightarrow A(B)$ obtained by restricting functions.  
Clearly, the pullback of the specialization map $\eta_k: A(B) \rightarrow \bar{\mathbb{Q}}_p$, $s \mapsto k$ under the composition of the two maps
above 
$$ \mathcal{O}_K[[T]] \rightarrow A(\mathfrak{B}^*) \rightarrow A(B)$$
is the specialization map $\mathcal{O}_K[[T]] \rightarrow \bar{\mathbb{Q}}_p$, $T \mapsto (1+p)^k-1$.

Now let $\mathcal{F} = \sum_{n=1}^\infty a_n q^n$ be a Coleman family specializing to $f$ at weight $k_0$, with $a_n$ rigid analytic functions on $\mathcal{U}$. If $\mathcal{U} = B$, 
then $a_n \in A(B)$, for all $n\ge 1$. For each $a_n \in A(B)$, assume there is a power series $f_n(T) \in \mathcal{O}_K[[T]]$  such that $a_n(s)=f_n((1+p)^s-1)$,
for all $s \in B$.  Write $f_n(T)=\sum_{n=0}^\infty c_\nu T^\nu$, where $c_\nu\in \mathcal{O}_K$. Then for all
%
%
integers $k$ of the form $k=k_0 + p^{\lceil r \rceil +t}a$, for some integers $a$ and $t\ge 0$ (so  $k\in \mathbb{Z} \cap B$), we have, for each $n \ge 1$, 
\begin{align*}
   |a_n(k)-a_n(k_0)| & = 
                         |f_n((1+p)^k-1)-f_n((1+p)^{k_0}-1)| \\
                    & = \left| \sum_{\nu = 0}^\infty  c_{\nu}[((1+p)^k-1)^\nu-((1+p)^{k_0}-1)^\nu] \right| \\
                    & \le \sup_{\nu \ge 1}|c_\nu| |(1+p)^k-(1+p)^{k_0}|\\
                    & \le |p(k-k_0)| \\
                    & \le |p^{t+1}|.
\end{align*}
Taking the infimum over all $r$, 
and recalling that $M_f$ is the non-negative integer defined in Definition~\ref{radius}, we obtain a system of Kummer congruences:
$$
  k \equiv k_0 \mod p^{M_f +t} \implies f_k \equiv f \mod p^{t+1},
$$
for all $t \ge 0$.

\section{Main result}
  \label{section main}

We now prove Theorem~\ref{main1}. We start with a  useful definition. 

\begin{defn}
  Let $\chi$ be a character of level $N$. If $f$ is an overconvergent form of integer weight $k$, level $Np$ and character $\chi\omega^{j}$, for some integer $j$, 
  then we say that $f$ has weight-character $\chi \omega^{j}\langle\langle\cdot \rangle \rangle^k$.
\end{defn}

\begin{thm}\label{main22}
Suppose $p \ge 5$ is a prime and $N$ is a positive integer such that $(N,p)=1$. Let $f\in S_k(N,\chi)$ be a classical eigenform of finite slope $\alpha$. 
Assume that a $p$-stabilization $f_k$ of $f$ has slope $\alpha$. Let $M_{{f_k}}$ 
be as in \defref{radius} and $\delta_{{f_k}}$ as in \eqref{delta}. 
Let $\kappa \in \{2,3, \dots, p^{M_{f_{k}}+ \delta_{f_k}}+1 \}$ be the 
unique integer such that $k \equiv 2-\kappa  \mod p^{M_{f_k}+\delta_{f_k}}$. Then there exists a classical eigenform $g \in S_l(N,\chi)$ of slope $\alpha + \kappa-1$ such that
$$
g \equiv\theta^{\kappa-1} f \mod  p.
$$
If $f$ is a newform, so is $g$. Moreover, there is a non-negative integer $M$ such that
the
weight $l$ of $g$ can be chosen to be any integer satisfying the following two conditions:
\begin{enumerate}
 \item [(i)] $l > 2\alpha +2\kappa,$
 \item[{(ii)}] $l =(k-2+\kappa)p^{M}+\kappa +n(p-1)p^{M}$, for any $n\in \mathbb{Z}$.
 \end{enumerate}
\end{thm}

\begin{proof}
It is well known that the $p$-stabilization $f_k$ satisfies $f_k \equiv f \mod p$, since the prime Fourier coefficient of both sides away from $p$ are equal,
and the $p$-th Fourier coefficients are either both units with the same reduction or of positive slope with reduction equal to zero.  
Let $\mathcal{F}$ be an overconvergent family of slope $\alpha$ as in Theorem~\ref{familyU} of radius $p^{-M_{f_k}}$ passing through $f_k$. 
Since $k \equiv 2-\kappa \mod p^{M_{f_k}}$, by \thmref{familyU}, we see that there exists an overconvergent eigenform $f_{2-\kappa}$ of tame level $N$, 
weight-character $\chi \omega^{k-2+\kappa} \langle\langle \cdot \rangle\rangle^{2-\kappa}$ and slope $\alpha$ in the family $\mathcal{F}$. 
Since $k \equiv 2-\kappa \mod p^{M_{f_k}+ \delta_{f_k}}$, where $\delta_{f_k}$ is as in \eqref{delta}, we have 
\begin{equation}
  \label{id}
  f \equiv f_k \equiv f_{2-\kappa} \mod p,
\end{equation} 
by \eqref{familycong}.
Now, consider the form $$g_\kappa=\theta^{\kappa-1} f_{2-\kappa}.$$ By \thmref{thetaover}, $g_{\kappa}$ is an overconvergent eigenform of tame level $N$ and weight-character $\chi \omega^{k-2+\kappa} \langle\langle \cdot \rangle\rangle^{\kappa}$.  Clearly $g_{\kappa}$ has slope $\alpha+\kappa-1$ because $f_{2-\kappa}$ has slope $\alpha$. 
Applying $\theta^{\kappa-1}$  to \eqref{id}, we obtain 
\begin{equation}\label{id2}
 g_{\kappa} \equiv \theta^{\kappa-1} f \mod p.
\end{equation}
Let  $\mathcal{G}$ be a family consisting of overconvergent eigenforms of slope $\alpha+\kappa-1$ and 
tame level $N$ such that $g_{\kappa}$ is a weight $\kappa$ specialization of the family $\mathcal{G}$ (cf. Theorem~\ref{familyU}).
Let $M_{g_\kappa}$ and $\delta_{g_\kappa}$ 
 be as in Definition~\ref{radius} and \eqref{delta} for the form $g_\kappa$, and set 
\begin{eqnarray}
  \label{M}
  M & := & M_{g_\kappa} +  \delta_{g_\kappa}.
\end{eqnarray}
Choose a weight $l \equiv \kappa \mod p^{M}$ and assume that 
\begin{equation}\label{cond}
\begin{aligned}
&{(i)} ~~~l > 2\alpha +2\kappa,\hspace{8cm}\\
&{(ii)}~~ l \equiv k-2+2\kappa \mod  (p-1),
\end{aligned}
\end{equation}
and let $g_l$ be an eigenform of weight $l$ in the family $\mathcal{G}$.
Note that the weight-character of $g_l$ is $\chi \omega^{k-2+2\kappa-l} \langle\langle \cdot \rangle\rangle^{l}$ which is equal to $\chi \langle\langle \cdot \rangle\rangle^{l}$ because of the second assumption above. Furthermore, in view of the first assumption, by the control theorem (\thmref{classicity}), we see that $g_l$ is a classical eigenform of weight $l$, tame level $N$ and character $\chi$. 
By \eqref{familycong} and \eqref{id2}, we get
\begin{equation}\label{id3}
g_l \equiv g_{\kappa} \equiv \theta^{\kappa-1} f \mod p.
\end{equation}
By the Chinese remainder theorem, any simultaneous solution $l$ of the congruence $(ii)$ in \eqref{cond} and the congruence 
$l \equiv \kappa \mod p^M$ 
will be of the form given in the statement of the theorem.

We claim that $g_l$ is the $p$-stabilization of a form $g$ of slope $\alpha + \kappa -1$. 
First note that the eigenform $g_l$ is $p$-old. Indeed, if $g_l$ is a $p$-new, by \cite[Theorem 4.6.17(ii)]{miyake}, we get 
\begin{equation*}\label{id4}
c_p^2=\chi(p)p^{l-2},
\end{equation*}
where $c_p$ denotes the $p$-th Fourier coefficient of $g_l$. It follows that the slope of $g_l$ is $(l-2)/2$, i.e.,
$$
l=2\alpha +2\kappa,
$$
which contradicts the first assumption on $l$ in \eqref{cond}. Therefore, the form $g_l$ is $p$-old and so is a $p$-stabilization of an eigenform  
$g(z)= \sum_{n= 1}^\infty b_nq^n \in S_l(N,\chi)$. We now show that $g$ has slope $\alpha+\kappa-1$.
We know that $g_l$ is either $g_{\alpha_p}$ or $g_{\beta_p}$, where
$$g_{\alpha_p}(z):=g(z)-\beta_pg(pz) ~~{\rm and~}~ g_{\beta_p}(z):= g(z)-\alpha_pg(pz)$$
and $\alpha_p$, $\beta_p$ are their $U_p$ eigenvalues given by the roots of $X^2-b_pX+\chi(p)p^{l-1}$. Clearly,
\begin{align}\label{sum}
 \alpha_p+\beta_p=b_p,
  \end{align}
    \begin{equation}\label{ab}
  \alpha_p\beta_p=\chi(p)p^{l-1}~ \Longrightarrow~  v(\alpha_p)+v(\beta_p)=l-1.
  \end{equation}
Since $c_p$ is either $\alpha_p$ or $\beta_p$ and $v(c_p)=\alpha+\kappa-1$, by using \eqref{ab} we conclude that
$$
v(\alpha_p)\ne v(\beta_p),
$$
otherwise we would get $(l-1)/2 =\alpha+\kappa-1$, which is not possible because of condition $(i)$ in \eqref{cond}. Therefore, there is no loss of generality in assuming 
that $v(\alpha_p)< v(\beta_p)$ and hence, by \eqref{ab}, we get
\begin{equation*}\label{ineq}
 v(\alpha_p)<\frac{l-1}{2} {\rm ~~and~~}v(\beta_p)>\frac{l-1}{2}.
\end{equation*}
We now claim that $g_l= g_{\alpha_p}$. If not, $g_l= g_{\beta_p}$, hence their $U_p$ eigenvalues satisfy $c_p=\beta_p$, which gives
$$
\alpha+\kappa-1=v(\beta_p)>\frac{l-1}{2}, 
$$
which again contradicts condition $(i)$ in \eqref{cond}. We thus get $g_l(z)=g_{\alpha_p}(z)=g(z)-\beta_pg(pz)$ which gives $c_p=\alpha_p$. Since $v(\alpha_p)<v(\beta_p)$, by using \eqref{sum}, we conclude that $v(b_p)=v(\alpha_p)=\alpha+\kappa-1$, proving that $g$ has slope $\alpha+\kappa-1$. 
Since $\beta_p \equiv 0 \mod p$, we have $g \equiv g_l \mod p$. 
If follows from this and \eqref{id3}, that $g \equiv \theta^{\kappa-1} f \mod p$, which completes the proof when $f$ is an eigenform. If $f$ is $N$-new, then
the forms $f_k$, $f_{2-\kappa}$, $g_\kappa$, $g_l$ and $g$ above are all $N$-new, completing the proof of the theorem. 
\end{proof}

As remarked in the Introduction, the theorem above implies Theorem~\ref{main1}, and so also Corollary~\ref{main}.

\section{Compatibility with Reductions of Galois representations}\label{comp} 
Fix a prime $p\ge 5$ and a positive integer $N$ such that $(p,N)=1$. 
Let $f=\sum_{n = 1}^\infty a_nq^n\in S_k(N)$ be a classical normalized  eigenform (with character $\chi = 1$) of slope $\alpha$,
having a $p$-stabilization  $f_k$ of slope $\alpha$. Assume also that $k \equiv 0 \mod p^{M_{f_k} + \delta_{f_k}}$ if $\alpha > 0$.
Let $g=\sum_{n = 1}^\infty b_nq^n \in S_l(N)$ be a normalized 
eigenform (with character $\chi = 1$) of slope $\alpha +1$ produced by \corref{main}, so 
\begin{eqnarray}
  \label{f vs g}
  \bar{\rho}_g \simeq \bar{\rho}_f \otimes \omega.
\end{eqnarray} 
For small slopes $\alpha+1$, the shape of the local Galois representation $\bar{\rho}_g|_{G_p}$ can be obtained in two ways: one by using \eqref{f vs g}, thereby reducing the problem 
to determining the shape of $\bar{\rho}_f|_{G_p}$ which is a form of smaller slope $\alpha$, and the other directly. 
Since the reductions $\bar{\rho}_f|_{G_p}$ are known 
for all slopes smaller than $2$,
we can compare these two methods to compute $\bar{\rho}_g|_{G_p}$, when $\alpha + 1 \in [1,2)$.  We will see that 
the computation of $\bar{\rho}_g|_{G_p}$ using these two methods is compatible in all cases (Sec. \ref{slopes in [0,1)}). 
When $\alpha + 1 \ge 2$,  we can also
use \eqref{f vs g} to produce new examples of reductions $\bar{\rho}_g|_{G_p}$ which have not, as far as we know, been shown to exist. We do this
in Sec. \ref{extrap} when  $\alpha + 1 \in [2,3)$.  Finally, in Sec. \ref{zigzag}, we recall that \eqref{f vs g} is compatible with the zig-zag conjecture of \cite{gha19}.

\subsection{Compatibility for $\alpha \in [0,1)$.} 
\label{slopes in [0,1)}

We divide our discussion into three cases.
Sec.~\ref{slope 0} treats the case $\alpha=0$, Sec.~\ref{slopes in (0,1)} the case $\alpha \in (0,1)$, excluding weights 
$k \nequiv 3 \mod (p-1)$ if $\alpha= \frac{1}{2}$, and finally Sec.~\ref{slope 1/2} the exceptional case
$\alpha= \frac{1}{2}$ and $k\equiv 3 \mod (p-1)$. Each section contains a table, whose columns we describe now.
In the first column, we write down the structure of the reductions of the local Galois representations attached to $f$ on the inertia subgroup $I_p$ 
using an appropriate reference.  Using \corref{main}, or more precisely \eqref{f vs g}, we immediately obtain the shape of the local Galois representation attached to $g$ on $I_p$. It is stated 
in the second column. Clearly, the slope $\alpha+1$ of $g$ lies in the interval $[1,2)$. \corref{main} shows that the weight $l$ of $g$ is of the form 
$$l = kp^M+2+n(p-1)p^M,$$ for 
any $n\in \mathbb{Z}$. 
This information is enough to compute the reduction $\bar{\rho}_g|_{I_p}$ directly, using the recent work of the first author and his collaborators. 
It is listed in the third column. In spite of the rather complicated behavior of the representations involved in the tables, the
representations listed in columns 2 and 3 match in all cases.

\subsubsection{Compatibility for $\alpha=0$.} 
  \label{slope 0}

In an unpublished letter to Serre \cite{del}, Deligne obtained the shape of $\bar{\rho}_f|_{I_p}$ when $v(a_p)=0$. It is stated in the first column of Table 1. By  
\eqref{f vs g}, we obtain the structure of 
$\bar{\rho}_g|_{I_p}$. This is listed in the second column. In the third column, we use \cite{bgr18} to directly compute the shape of $\bar{\rho}_g|_{I_p}$.
To do this we need some notation. Set $r = l - 2 = kp^M+np^M(p-1)$ and suppose that $r \equiv b \mod (p-1)$ for the set of 
representatives $2 \le b \le p$ modulo $(p-1)$.  Also, if $b = 2$, set 
\begin{eqnarray}
  \label{tau t slope 1}
  t' =v(l-4) & \text{and} & \tau' = v \left( \frac{b_p^2 - {l-2 \choose 2}p^2}{pb_p} \right) \ge  0.
\end{eqnarray}
Using \cite{bgr18}, we obtain column 3 of Table 1.
\vspace{.3cm}

\begin{tabular}{ |p{4.3cm}||p{2.6cm}|p{4.7cm}|  }
 \hline
Deligne \cite{del} & \corref{main} &Bh-Gh-Ro \cite{bgr18}\\
 \hline
$f\in S_k(N)$ with $v(a_p)=0$  &\multicolumn{2}{|c|}{$g\in S_l(N)$ with $v(b_p)=1$ \qquad \qquad \qquad \qquad}\\
 \hline
$~$ \newline \newline \newline \newline 
$\bar{\rho}_f|_{I_p}  \simeq\omega^{k-1} \oplus 1.$  
& $~$ \newline \newline \newline \newline   $\bar{\rho}_g|_{I_p}\simeq \omega^{k} \oplus \omega $. & 
{{\bf Case (i)}:} $b=2$. \newline
$M\ge 1 \implies \tau' = 0 = t'$. \newline 
$\Longrightarrow \bar{\rho}_g|_{I_p} \simeq \omega^b \oplus \omega $. \newline \newline 
{{\bf Case (ii)}:} 
$b=3, \dots, p-1$. \newline
 $M\ge 1 \implies p\nmid r-b$.  \newline
$\Longrightarrow \bar{\rho}_g|_{I_p} \simeq \omega^b \oplus \omega $. \newline \newline
{{\bf Case (iii)}:} $b=p$. \newline $M\ge 1 \implies p\mid r-b$.  \newline
$\Longrightarrow \bar{\rho}_g|_{I_p} \simeq \omega \oplus \omega $. \\ 
 \hline 
\end{tabular}
\\
\begin{center}
{   Table 1. ($\alpha=0$)} 
\end{center}
Since $b \equiv k$ mod $(p-1)$, the reductions $\bar{\rho}_g|_{I_p}$  listed in columns 2 and 3 of Table 1 match.

\subsubsection{Compatibility for $\alpha\in (0,1)$ with $k\nequiv 3 \mod (p-1)$ if $\alpha= \frac{1}{2}$.} 
  \label{slopes in (0,1)}

In Table 2 we compare the reductions obtained in \cite{buzgee1} for slopes in $(0,1)$ and \cite{bg15} for
slopes in $(1,2)$. These papers do not treat completely  the difficult cases of exceptional weights $k \equiv 3 \mod (p-1)$ if $\alpha = \frac{1}{2}$ and 
$l \equiv 5 \mod (p-1)$ if $\alpha = \frac{3}{2}$, respectively, but we shall deal with them in the following section. To use \cite{buzgee1}, let 
$t-1$ be the residue class of $k-2$ modulo $(p-1)$, for $1\le t \le p-1$. 
To use \cite{bg15}, we let $r = l-2$ and $b$ be as in Table 1. We obtain Table 2. 
\vspace{.3cm}

\begin{tabular}{ |p{5.2cm}||p{3.8cm}|p{4.4cm}|  }
 \hline
Buzzard-Gee \cite{buzgee1} & \corref{main} & Bh-Gh \cite{bg15}\\
 \hline
$f\!\in\! S_k(N)$ with $v(a_p)\in (0,1)$,  &\multicolumn{2}{|c|}{$g\in S_l(N)$ with $v(b_p) \in (1,2)$,}\\
 $v(a_p)\neq \frac{1}{2}$ if $k\equiv 3 \mod (p-1)$ & \multicolumn{2}{|c|}{$v(b_p)\neq \frac{3}{2}$ if $l \equiv 5 \mod (p-1)$} \\
 \hline
$~$ \newline \newline \newline \newline \newline
$\bar{\rho}_f|_{G_p} \simeq {\rm ind}\left(\omega_2^{t}\right).$
&$~$ \newline \newline  \newline \newline \newline   
$\bar{\rho}_g|_{G_p} \simeq {\rm ind} \left( \omega_2^{t+p+1} \right).$
& {{\bf Case (i)}:} $b=2$. \newline 
$M\ge 1 \implies p\mid r(r-1).$ \newline 
$\Longrightarrow \bar{\rho}_g|_{G_p} \simeq {\rm ind}\left(\omega_2^{b+p}\right)$. \newline \newline 
{{\bf Case (ii)}:} $b= 3, \dots, p-1$. \newline 
$M\ge 1 \implies p\nmid r-b$.  \newline
$\Longrightarrow \bar{\rho}_g|_{G_p} \simeq {\rm ind}\left(\omega_2^{b+p}\right) $. \newline \newline
{{\bf Case (iii)}:} $b=p$. \newline 
$M\ge 2 \implies p^2 \nmid r-b$.  \newline
$\Longrightarrow \bar{\rho}_g|_{G_p} \simeq {\rm ind}\left(\omega_2^{2p}\right) $. \\ 
\hline 
 \end{tabular}
\\
 \begin{center}
{   Table 2. ($\alpha\in (0,1)$ with $k\nequiv 3 \mod (p-1)$ if $\alpha= \frac{1}{2}$) } 
\end{center}
Note $l \equiv k+2 \mod (p-1)$ implies $b = t+1$, so the shape of $\bar{\rho}_g|_{G_p}$ in columns 2 and 3 of Table 2 match.

\subsubsection{Compatibility for $\alpha= \frac{1}{2}$ and $k\equiv 3 \mod (p-1)$. } 
  \label{slope 1/2}

In Table 3, we compare the results of \cite{buzgee2} which treats the exceptional weights $k\equiv 3 \mod (p-1)$ if the slope is $\frac{1}{2}$ and 
the recent work \cite{gr19} which treats the exceptional weights $l \equiv 5 \mod (p-1)$ when the slope is $\frac{3}{2}$. 
In order to use these works we introduce the following notation. Set 
\begin{eqnarray}
  t= v(k-3) & \text{and} & \tau = v \left( \frac{a_p^2-(k-2)p}{pa_p} \right) \ge -\frac{1}{2}, \label{tau t slope 1/2}\\ 
  t'=v(l-5) & \text{and} & \tau'= v\left( \frac{b_p^2-(l-4) {l-3 \choose 2}p^3}{pb_p} \right) \ge  \frac{1}{2}. \label{tau t slope 3/2}
\end{eqnarray}
As above, we obtain Table 3. 
\vspace{.3cm}

 \begin{tabular}{ |p{4.3cm}||p{4.0cm}|p{4.7cm}|  }
 \hline
Buzzard-Gee \cite{buzgee2}& \corref{main} &Gh-Ra \cite{gr19}\\
 \hline
$f\in S_k(N)$ with $v(a_p)\!=\!\frac{1}{2}$ and $k\equiv 3 \mod (p-1)$   &\multicolumn{2}{|c|}{$g\in S_l(N)$ with $v(b_p)=\frac{3}{2}$ and $l\equiv 5 \mod (p-1)$} \\
\hline 
$~$ \newline 
$
\bar{\rho}_f|_{I_p} \simeq \newline \newline \begin{cases}
              \omega_2^{2} \oplus \omega_2^{2p}, &  \tau< t,\\
              \omega \oplus \omega, & \tau \ge t.
                            \end{cases}$
\newline
&$~$ \newline    
$\bar{\rho}_g|_{I_p} \simeq \newline \newline \begin{cases}
              \omega_2^{p+3} \oplus \omega_2^{3p+1}, &  \tau< t,\\
              \omega^2 \oplus \omega^2, & \tau \ge t.
                            \end{cases}$ \newline
&
$M \ge 1 \implies t' = 0$. \newline
$\Longrightarrow
\bar{\rho}_g|_{I_p} \simeq  \newline \newline \begin{cases}
              \omega_2^{p+3} \oplus \omega_2^{3p+1}, & \frac{1}{2} \le \tau'< 1,\\
              \omega^2 \oplus \omega^2, & \tau' \ge 1.
                            \end{cases}$ \newline \\
 \hline
  \end{tabular}
\\
 \begin{center}
{   Table 3. ($\alpha= \frac{1}{2}$ and $k\equiv 3 \mod (p-1)$)} 
\end{center}
Again, we see that the shapes of  $\bar{\rho}_g|_{I_p}$ in columns 2 and 3 of Table 3 are compatible. In fact, Table 3 shows that if $\tau < t$  (respectively, $\tau \ge t$), then we must have
$\tau' < 1$ (respectively, $\tau' \ge 1$).

\subsection{Extrapolation of the shape of $\bar{\rho}_g|_{G_p}$ }\label{extrap}
Let $f$ be an eigenform of slope $\alpha \in [1,2)$ as in Corollary~\ref{main}. 
Using the results of \cite{bgr18}, \cite{bg15} and \cite{gr19} and \eqref{f vs g}, we see that there is an eigenform $g  = \sum_{n=1}^\infty b_n q^n$ of weight $l$, level coprime to $p$, 
and slope $v(b_p) = \alpha + 1 \in [2,3)$ such that $\bar{\rho}_g|_{I_p}$ has one of the following structures, although of course there may be other structures (e.g., those not coming
from the theta operator). 
\vspace{.3cm}

{\noindent \bf Case (i)} $v(b_p)=2$.
$$
\bar{\rho}_g|_{I_p} \simeq \begin{cases}
              {\rm ind}\left(\omega_2^{b+p+2}\right), & b=2,3,\dots, p-1,\\
               \omega^{b+1} \oplus \omega^2, &  b=2,3,\dots, p,\\
                {\rm ind}\left(\omega_2^{b+2p+1}\right), & b=2,p,\\
                 \end{cases} $$
where $2\le b \le p$ is an integer such that $b \equiv l-4 \mod (p-1)$. 
\vspace{.3cm}

{\noindent \bf Case (ii)} $v(b_p) \in (2,3)$ and $l \nequiv 7 \mod (p-1)$ if $v(b_p) = \frac{5}{2}$.
$$
\bar{\rho}_g|_{I_p} \simeq \begin{cases}
              {\rm ind}\left(\omega_2^{b+p+2}\right), & b=2,3,\dots, p-1,\\
              {\rm ind}\left(\omega_2^{b+2p+1}\right), & b=2,3\dots,p,\\
              \omega^{2} \oplus \omega^2, &  b=p,\\
                 \end{cases} $$
where  $2\le {b}\le p$ is an integer such that ${b}\equiv l-4 \mod (p-1)$. 
\vspace{.3cm}

{\noindent \bf Case (iii)} $v(b_p) = \frac{5}{2}$ if $l \equiv 7 \mod (p-1)$.
$$
\bar{\rho}_g|_{I_p} \simeq \begin{cases}
              {\rm ind}\left(\omega_2^{p+5}\right), \\
               \omega^{4} \oplus \omega^2, \\
                {\rm ind}\left(\omega_2^{2p+4}\right), \\
                 \omega^{3} \oplus \omega^3.
                 \end{cases} $$

\subsection{Zig-zag conjecture}\label{zigzag}
Let $f\in S_k(N)$ be a normalized eigenform of half-integral slope $\alpha$ such that $0<\alpha \le \frac{p-1}{2}$. Then $f$ is said to have exceptional weight $k$ 
for slope $\alpha$ if $k\equiv 2\alpha+2\mod (p-1)$. If $b= 2\alpha$, then the first author conjectured that in the general exceptional case, there are $b+1$ possibilities for 
the reduction $\bar{\rho}_f|_{I_p}$ with various irreducible and reducible cases occurring alternately (in fact, the conjecture is for general crystalline representations). The precise 
version is outlined in a conjecture called the zig-zag conjecture \cite[Conjecture 1.1]{gha19}. It is known that the zig-zag conjecture holds for exceptional weights corresponding 
to slopes $\frac{1}{2}$, $1$ and $\frac{3}{2}$ (by  \cite{buzgee2, bgr18, gr19}, respectively). Tables 1 and 3 show that these known cases of the
zig-zag conjecture are compatible with the theta operator (more precisely, Corollary~\ref{main}). In \cite[Sec. 4.2]{gha19}, the first author showed 
that the general zig-zag conjecture is compatible with the theta operator.
See \cite{gha19} for further details.

\section{Upper bounds for the radii of Coleman families}\label{lb}

Let $f \in S_k(N)$ and $g \in S_l(N)$ be two normalized eigenforms as in \corref{main} (with trivial character $\chi = 1$)  and 
slopes $\alpha$ and $\alpha +1$, respectively.  Recall that $p^{-M_{g_2}}$ is a radius for a Coleman 
family passing through $g_2 = \theta f_0$ (see the proof of \thmref{main22}), and that the integer
$M = M_{g_2} + \delta_{g_2}$ defined in \eqref{M} appears in the formula for the weight $l$ of $g$, namely $l = k p^M + 2 +n(p-1)p^M$, for
any $n \in \mathbb{Z}$. Assume that $M_{g_2} \neq r_{g_2}$, i.e., $\delta_{g_2} = 0$, 
so that $$M = M_{g_2}.$$ In this section, we obtain lower bounds for $M_{g_2}$ (so upper bounds for the radii of the Coleman 
family passing through $g_2$)
when $\alpha \in [0,1)$ using the compatibility 
results in Secs. \ref{slope 0},  \ref{slopes in (0,1)} and \ref{slope 1/2}, regarding
the reductions of Galois representations of slopes $\alpha$ and $\alpha + 1$. 
\vspace{.2cm}

{\noindent \underline {Case (1)}:} $\alpha=0$. In this case $g = \sum_{n=1}^\infty b_n q^n$ has slope 1. 
Let $\tau'$ and $t'$ be the parameters defined in \eqref{tau t slope 1} when $l \equiv 4 \mod (p-1)$. 
\begin{prop}
\label{0}
Assume $g$ has slope $1$. If $l \nequiv 4 \mod (p-1)$, then $M \ge 1$. If $l \equiv 4 \mod (p-1)$, then $\tau' = t'$.
\end{prop}
\begin{proof}
 Recall $r= l-2 =  kp^M+np^M(p-1)$, and $b\equiv r  \equiv k \mod (p-1)$ where $2 \le b \le p$
 is the representative of $r$ modulo $(p-1)$. 
 First assume $b = 3, \dots p-1$, which is {\bf Case (ii)} of Table 1. 
 Suppose $M=0$. Choosing $n \equiv k-b \mod p$, we see that $r - b \equiv 0 \mod p$, 
 so the main theorem of  
 \cite{bgr18} shows that 
 $\bar{\rho}_g|_{G_p} \simeq {\rm ind} (\omega_2^{b+1} )$ is irreducible, whereas
 the middle column of Table 1 (which comes from Corollary~\ref{main}) yields that 
 $\bar{\rho}_g|_{I_p} \simeq \omega^k \oplus \omega$. This is a contradiction. Hence $M \ge
1$. Now assume $b = p$, which is {\bf Case (iii)} of Table 1. If $M = 0$, choosing $n \nequiv k \mod p$, we have $r - b \equiv k-n \nequiv 0 \mod p$, 
so by \cite{bgr18},  $\bar{\rho}_g|_{G_p} \simeq {\rm ind} (\omega_2^{b+p} )$ is irreducible,
another contradiction. Finally, assume $b = 2$. If $\tau' \neq t'$, \cite{bgr18} shows that  
$\bar{\rho}_g|_{G_p} \simeq  {\rm ind} (\omega_2^{b+1})$ or $ {\rm ind} (\omega_2^{b+p})$, both of which are irreducible, giving a contradiction. 
\end{proof}

{\noindent \underline {Case (2)}:} $\alpha\in (0,1)$ and $\alpha \neq \frac{1}{2}$ if 
$k\equiv 3 \mod (p-1)$. 

\begin{prop}\label{12}
If $g$ has slope in $(1,2)$ and $l \nequiv 5 \mod (p-1)$ if the slope is $\frac{3}{2}$, then
$$
M \> \ge \>
\begin{cases}
 2, &  \text{if } l \equiv 3 \mod (p-1),\\
 1, &  \text{otherwise}.
\end{cases}
$$
\end{prop}
\begin{proof}
Let notation be as in the proof of the previous proposition.
First assume $l \equiv 3 \mod (p-1)$, so $b=p$. This is {\bf Case (iii)} of Table 2. 
If $M=1$, then choosing $n \equiv k-1 \mod p$,
gives $r-b \equiv 
 0 \mod p^2$. Applying the main theorem of \cite{bg15}, we see that $\bar{\rho}_g|_{I_p} \simeq \omega
\oplus \omega$ is reducible whereas from the middle column of Table 2, 
it is ${\rm ind} (\omega_2^{b+p})$, so irreducible (on $G_p$), giving a contradiction. If $M = 0$, choosing $n
\equiv k-p+kp \mod p^2$ 
leads to the same contradiction. Therefore, $M\ge 2$. If $l \nequiv 3 \mod (p-1)$, then we are in {\bf Case (i)} and {\bf (ii)}
of Table 2. If $M=0$, then choosing $n \equiv k-b \mod p$, we see that $p \nmid r(r-1)$ if $b =2$ and
$p \mid r-b$ if $b = 3, \dots, p-1$. Either way, Bhattacharya and Ghate  \cite{bg15} show that $\bar{\rho}_g|_{G_p}
\simeq {\rm ind} (\omega_2^{b+1} )$ is a {\it different} irreducible representation, a contradiction,
so $M \ge 1$. Note that since we have excluded the exceptional weights $l \equiv 5 \mod (p-1)$
when the slope of $g$ is $\frac{3}{2}$, the main result of \cite{bg15} does indeed apply.
\end{proof}

{\noindent \underline {Case (3)}:} $\alpha= \frac{1}{2}$ and $k\equiv 3 \mod (p-1)$. 
In this case $g = \sum_{n=1}^\infty b_n q^n$ has slope $\frac{3}{2}$. Let $\tau'$ and $t'$ be as in 
\eqref{tau t slope 3/2}. 
In this case we are only able to prove the following result.
\begin{prop}\label{1/2}
  If $g$ has slope $\frac{3}{2}$ and $l\equiv 5 \mod (p-1)$, then $\tau' > t'$.
\end{prop}

\begin{proof}
 The proof is similar to the proof of the last case of Propositions \ref{0}. If $\tau' \le t'$, then \cite{gr19} shows
 that $\bar{\rho}_g|_{I_p} \simeq \omega_2^{4} \oplus \omega_2^{4p}$ or $\omega^3 \oplus \omega$, neither of which 
 occur in the middle column of Table 3, giving a contradiction. 
\end{proof}

\begin{cor}
  In the cases of slopes $\alpha \in [0,1)$ treated above, the least $M$ we can take as one varies over all $g$ of slope $\alpha+1$ coming from \corref{main} is at least 
  $\lceil \alpha + 1 \rceil$. 
\end{cor}

In \cite[Theorem 7.4, Remark 6.8]{ber17}, Bergdall obtains general lower bounds for $M$ via a similar analysis involving the 
reductions of crystalline representations, but using \cite{blz} instead of \cite{bgr18}, \cite{bg15}, \cite{gr19}. However, we remark that 
the bound he obtains is trivial, i.e., $M \ge 0$, when the slope is small (e.g., smaller than 2).

\section{Slope of the form obtained from Serre's conjecture}\label{serexmp} 

Let $p$ be an odd prime in this section. Let $f \in S_k(N, \chi)$ be an eigenform of weight $k \ge 2$, level $N$ coprime to $p$, character $\chi$ 
and finite slope $\alpha$ such that $\bar\rho_f$ is irreducible.  
Consider the twisted representation $\bar{\rho}_f \otimes \omega$. It is known that 
$\rho_f |_{I_p} \simeq V_{k,a_p\sqrt{\chi(p)}}$, by \cite[Theorem 6.2.1]{bre}. By Serre's conjecture (proved in \cite{kha}, \cite{khw}, \cite{kis}),
there is a normalized eigenform $h$ of minimal weight  (Serre weight)
$k(\bar{\rho}_h)$ satisfying $$\bar{\rho}_h \simeq \bar{\rho}_f \otimes \omega.$$
The aim of this section is to show that the slope of $h$ is not necessarily $\alpha+1$. 
In order to show this we give two examples.

\vspace{.2cm}

{\noindent \bf Example (1):} Assume the form $f$ has slope $\alpha=0$ so that $\bar{\rho}_f|_{I_p} \simeq \left(\begin{smallmatrix} \omega^{k-1} & * \\ 0 & 1 \end{smallmatrix}\right)$.
Suppose that $k \equiv 0 \mod (p-1)$. 
Then
\begin{equation}\label{id7}
  \bar{\rho}_h|_{I_p}=\bar{\rho}_f \otimes \omega |_{I_p} \simeq \left(\begin{smallmatrix} 1 & * \\ 0 & \omega \end{smallmatrix}\right).
\end{equation}
By using Serre's recipe for the Serre weight (see \cite{edi}), we see that 
$$
k(\bar{\rho}_h)=\begin{cases} 
                    2 & \text{if } * = 0, \\
                   2p & \text{if } * \neq 0.
                \end{cases}
$$
In the latter case ($* \neq 0$), if the slope of $h$ is $1$, then by \cite{bre}, $\bar{\rho}_h|_{I_p} \simeq \left(\begin{smallmatrix} 1 & 0\\ 0 & \omega \end{smallmatrix}\right)$, which is compatible with \eqref{id7}.
However, in the former case ($* = 0$), if the slope of $h$ is $1$, then by \cite{edi}, we would have
$\bar{\rho}_h|_{I_p} \simeq \left(\begin{smallmatrix} \omega_2 & 0\\ 0 & \omega_2^p\end{smallmatrix}\right)$ which contradicts \eqref{id7}.
So we see that if $* = 0$, the form $h$ cannot have slope $1$. Forms $f$ of slope $0$ with $* = 0$ are rare (the vanishing is related
to the existence of companion forms \cite{gro}), but as an example of a form $f$ satisfying this vanishing condition and all the 
hypotheses of this section, one may take $p = 3$ and the form $f$ corresponding to the elliptic
curve of conductor $N = 89$.

\vspace{.2cm}

{\noindent \bf Example (2):} Assume the form $f$ has slope $\frac{1}{2}$. Suppose $3  < k\equiv 3 \mod (p-1)$ and $\tau \ge t$,
where $\tau$ and $t$ are defined in \eqref{tau t slope 1/2}. Then by \cite{buzgee2}, we have 
$\bar{\rho}_f|_{I_p} \simeq \omega \oplus \omega$ up to semisimplification, and so
\begin{equation}\label{id8}
\bar{\rho}_h|_{I_p} \simeq \omega^2 \oplus \omega^2,
\end{equation}
up to semisimplification. A Serre weight computation gives 
$$
k(\bar{\rho}_h)=2p+3,
$$
independently of whether $\bar\rho_h |_{I_p}$ is semi-simple or not. 
Now assume that $h$ has slope $\frac{3}{2}$. The shape of $\bar\rho_h |_{G_p}$ has been recently worked out in \cite{gr19}. With notation as in Sec. \ref{comp},
we have $r=2p+1$, so $r \equiv 3 \mod (p-1)$, so $b = 3$. Then $\tau' = \frac{1}{2}$ and $t'=0$, where $\tau'$ and $t'$ are defined
by \eqref{tau t slope 3/2} with $b_p$ the
$p$-th Fourier coefficient of $h$. By the main theorem of \cite{gr19}, we have  
$\bar{\rho}_h|_{I_p} \simeq \omega_2^{p+3} \oplus \omega_2^{3p+1}$ which contradicts \eqref{id8}. Thus $h$ cannot have slope $\frac{3}{2}$.
It should be possible to produce a numerical example of a form $f$ satisfying all the hypotheses of this example.

\begin{acknowledgements}
This work was carried out while the second  author was a postdoctoral fellows at the Tata Institute of Fundamental Research, Mumbai.
\end{acknowledgements}

\end{document}